\documentclass[12pt,oneside, reqno]{amsart}
\usepackage[margin=1in]{geometry}  

      \usepackage{amssymb}

\usepackage{ulem}       
\usepackage{amsmath, amsthm, amssymb}
\usepackage{xspace}
\usepackage{mathrsfs}

 \usepackage{graphicx}  
 \usepackage{color}	
 \usepackage{mathabx}
\newtheorem{theorem}{Theorem}[section]

\newtheorem{lemma}[theorem]{Lemma}

\newtheorem{example}[theorem]{Example}
\newtheorem{question}{Problem} 

\newcommand{\e}{\eta}
\newcommand{\scr}{\mathscr}

      \makeatletter
      \def\@setcopyright{}
      \def\serieslogo@{}
      \makeatother

\begin{document}

   \author{Amin Bahmanian}
   \address{Department of Mathematics and Statistics
 Auburn University, Auburn, AL USA   36849-5310}
   \email{mzb0004@tigermail.auburn.edu}

   \author{  C. A. Rodger}

   \address{Department of Mathematics and Statistics
 Auburn University, Auburn, AL USA   36849-5310}

   \email{rodgec1@auburn.edu}

   \title[What are graph amalgamations?]{What are graph amalgamations?}

   \begin{abstract}
     In this paper, a  survey about recent progress on problems solved using graph amalgamations is presented, along with some new results with complete proofs, and some related open problems.
   \end{abstract}


   \dedicatory{Dedicated to Lucia Gionfriddo, a young star, the memory of whom will
shine brightly in the minds of we who knew her.}

   \date{\today}

   \maketitle
\section{Introduction} 
Edouard Lucas (1842--1891), the inventor of the Towers of Hanoi problem, discussed the \textit{probl\'eme de ronde } that asked the following \cite{L}:
Given $2n+1$ people, is it possible to arrange them around a single table on $n$ successive nights so that nobody is seated next to the same person on either side more than once? This problem is equivalent to a Hamiltonian decomposition of $K_{2n+1}$; that is partitioning the edge set of $K_{2n+1}$ into spanning cycles.  A solution to this problem for $n=3$ is illustrated in Figure \ref{walecki}, which is due to Walecki. This can be easily generalized to any complete graph by ``rotating'' an initial cycle. 
\begin{figure}[htbp]
\begin{center}
\scalebox{.45}{ \input {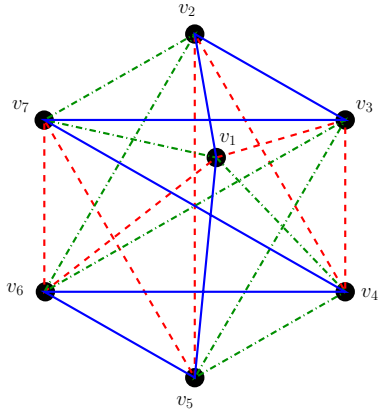} }
\caption{Walecki Construction  }
\label{walecki}
\end{center}
\end{figure} 

In 1984, Hilton \cite{H2} suggested a different approach to solving this problem, one of which is useful for solving another family of problems as well. He first fused  all the vertices  of $K_n$ (this is called amalgamation)  which results in having $\binom{n}{2}$ loops incident with a vertex. Then he shared the loops evenly between different color classes. (In this paper, the $i^{th}$ color class of $G$ is defined to  be the spanning subgraph of $G$ that contains precisely the edges colored $i$.) Finally he reversed the fusion by splitting the single vertex into $n$ vertices (this is called detachment), so that each color class is a Hamiltonian cycle. This is illustrated in Figure \ref{hilton} for $K_7$. 
\begin{figure}[htbp]
\begin{center}
\scalebox{.55}{ \input {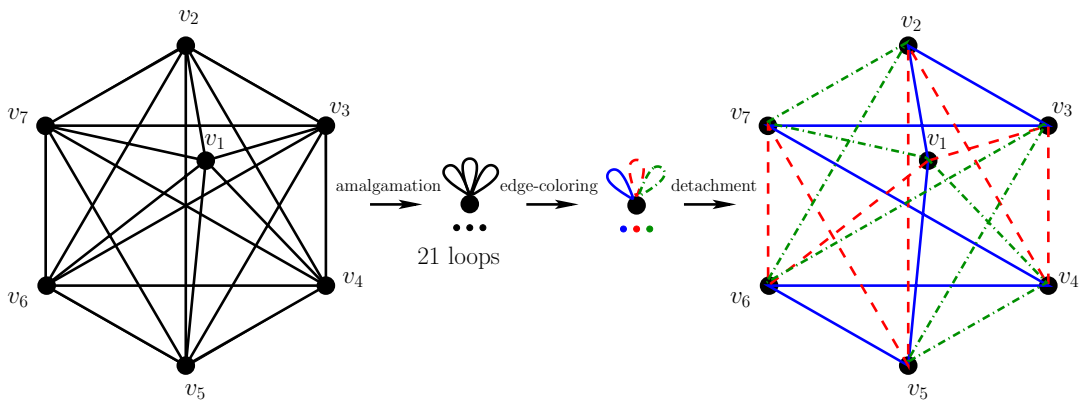} }
\caption{Hamiltonian decomposition of $K_7$ }
\label{hilton}
\end{center}
\end{figure} It is not obvious how we can detach the loops so that each color class is a Hamiltonian cycle. The second problem that Hilton solved was an embedding problem \cite{H2}. Given an edge-coloring of $K_m$, in which each color class is a path, he used amalgamations to extend this coloring to  an edge-coloring of $K_{m+n}$, so that each color class is a Hamiltonian cycle in $K_{m+n}$ (so $m+n$ must be odd). The idea is to add a new vertex, say $u$ to $K_m$ incident with $\binom{n}{2}$ loops so that there are $n$ edges between this vertex and every other vertex. Let us call this graph $K_m^+$. (In fact $K_m^+$ is an amalgamation of $K_{m+n}$ in which all further $n$ vertices are contracted in one point.) One can easily color all the edges incident with $u$ so that the valency of $u$ for each color class is exactly $2n$. Finally by detaching $u$ into $n$ vertices, say $u_1,\ldots,u_n$ and sharing the edges of each color class incident with $u$ among $u_1,\dots, u_n$ as evenly as possible and ensuring that each color class is connected, provides the desired outcome: a Hamiltonian decomposition of $K_{m+n}$. 
This is illustrated for $m=5,n=2$ in Figure \ref{hiltonembed}. 
\begin{figure}[htbp]
\begin{center}
\scalebox{.45}{ \input {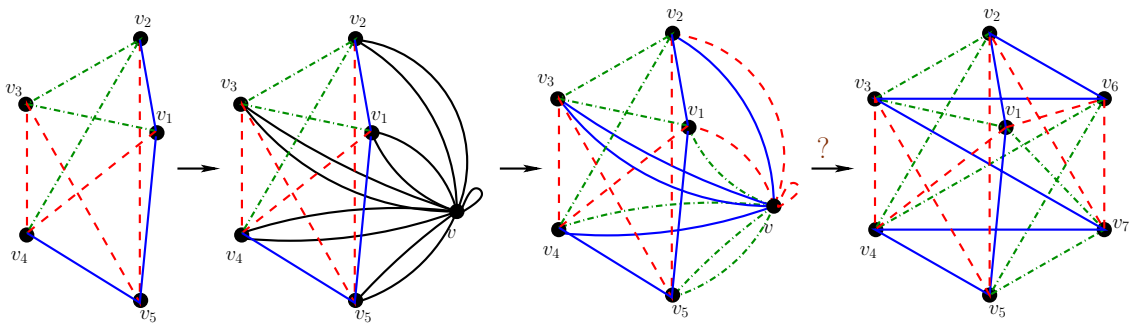} }
\caption{Embedding a path decomposition of $K_5$ into a Hamiltonian decomposition of $K_7$}
\label{hiltonembed}
\end{center}
\end{figure} To provide more explanation, first we give some definitions. 

A graph $H$ is an \textit{amalgamation} of a graph $G$  if, there exists a function $\phi$ called an \textit{amalgamation function} from $V(G)$ onto $V(H)$ and a bijection $\phi':E(G)\rightarrow E(H)$ such that $e$ joining $u$ and $v$ is in $E(G)$ if and only if $\phi'(e)$ joining $\phi(u)$ and $\phi(v)$ is in $E(H)$; we write $\phi(G)=H$. In particular, this requires that $e$ be a loop in $H$ if and only if, in $G$, it either is a loop or joins distinct vertices $u,v$, such that $\phi(u) = \phi(v)$. 
(Note that $\phi'$ is completely determined by $\phi$.) Associated with $\phi$ is the \textit {number function}  $\eta:V(H)\rightarrow \mathbb N$ defined by $\eta(v)=|\phi^{-1}(v)|$, for each $v\in V(H)$. We also shall say that $G$ is a \textit{detachment} of $H$ in which each vertex $v$ of $H$ splits (with respect to $\phi$) into the vertices in $\phi^{-1}(\{v\})$ (see Figure \ref{k22221hd1label}). 
A \textit{detachment} of $H$ is, intuitively speaking, a graph obtained from $H$ by splitting some or all of its vertices into more than one vertex. If $\eta$ is a function from $V(H)$ into $\mathbb {N}$ (the set of positive integers), then an \textit{$\eta$-detachment} of $H$ is a detachment of $H$ in which each vertex $u$ of $H$ splits into $\eta(u)$ vertices. In other words, $G$ is an $\eta$-detachment  of $H$ if  there exists an amalgamation function $\phi$ of $G$ onto $H$ such that $|\phi^{-1}(\{u\})|=\eta(u)$ for every $u\in V(H)$. 
\begin{figure}[htbp]
\begin{center}
\scalebox{.35}{ \input {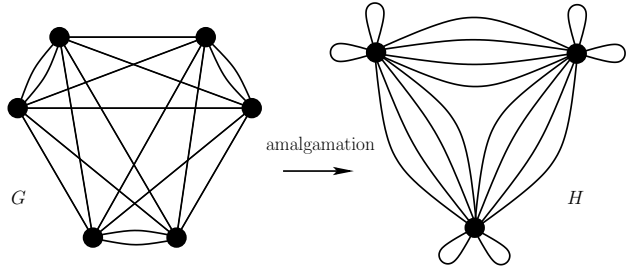} }
\caption{A graph $G$ with one of its amalgamations $H$ }
\label{k22221hd1label}
\end{center}
\end{figure} 
In this paper $x\approx y$ means $\lfloor y \rfloor \leq x\leq \lceil y \rceil$, $\ell(u)$ denotes the number of loops incident with vertex $u$, $d(u)$ denotes the degree of vertex $u$ (loops are considered to contribute two to the degree of the incident vertex), the subgraph of $G$ induced by the edges colored $j$ is denoted by $G(j)$, $\omega (G)$ is the number of components of $G$, and the \textit {multiplicity} of a pair of vertices  $u,v$ of $G$, denoted by $m(u,v)$, is the number of edges joining $u$ and $v$ in $G$. A \textit{k-edge-coloring} of a graph $G$ is a mapping $f:E(G)\rightarrow C$, where $C$ is a set of $k$ \textit{colors} (often we use $C=\{1,\ldots,k\}$). It is often convenient to have empty color classes, so we do not require $f$ to be surjective.

One of the most useful properties that one can obtain using the techniques described here, is that many graph parameters (such as colors, degrees, multiple edges) can be simultaneously shared evenly during the detachment process. This is often the most desirable property. 

\begin{theorem}\textup{(Bahmanian, Rodger \cite[Theorem 1]{BahRod1})}\label{mainth}
Let $H$ be a $k$-edge-colored graph  and let $\eta$ be a function from $V(H)$ into $\mathbb{N}$ such that for each $v \in V(H)$, $\eta (v) = 1$ implies $\ell_H (v) = 0$. Then there exists a loopless $\eta$-detachment $G$ of $H$  in which each $v\in V(H)$ is detached into $v_1,\ldots, v_{\eta(v)}$, such that $G$ satisfies the following conditions:
\begin{itemize}
\item [\textup{(A1)}] $d_G(u_i) \approx d_H(u)/\eta (u)  $ for each $u\in V(H)$ and  $1\leq i\leq \eta(u);$
\item [\textup{(A2)}] $d_{G(j)}(u_i) \approx d_{H(j)}(u)/\eta (u)  $ for each $u\in V(H)$,  $1\leq i\leq \eta(u)$, and $1\leq j\leq k;$
\item [\textup{(A3)}] $m_G(u_i, u_{i'}) \approx \ell_H(u)/\binom {\eta(u)}{2} $ for each $u\in V(H)$ with $\eta (u) \geq 2$ and $1\leq i<i'\leq \eta(u);$
\item [\textup{(A4)}] $m_{G(j)} (u_i, u_{i'}) \approx  \ell_{H(j)}(u)/\binom{\eta (u)}{2}  $ for each $u\in V(H)$ with $\eta (u) \geq 2$, $1\leq i<i'\leq \eta(u)$, and $1\leq j\leq k;$
\item [\textup{(A5)}] $m_G(u_i, v_{i'}) \approx m_H(u, v)/(\eta (u) \eta (v)) $ for every pair of distinct vertices $u,v\in V(H)$, $1\leq i\leq \eta(u)$, and $1\leq i'\leq \eta(v);$
\item [\textup{(A6)}] $m_{G(j)}(u_i, v_{i'}) \approx m_{H(j)}(u, v)/(\eta (u)\eta (v))  $ for every pair of distinct vertices $u,v\in V(H)$, $1\leq i\leq \eta(u)$, $1\leq i'\leq \eta(v)$, and $1\leq j\leq k;$
\item [\textup{(A7)}] If for some $j$, $1\leq j\leq k$, $d_{H(j)}(u)/\eta (u)$ is even for each $u \in V(H)$, then $\omega(G(j)) = \omega(H(j))$.
\end{itemize}
\end{theorem}
The proof uses edge-coloring techniques.  An edge-coloring of a multigraph is (i) \textit{equalized} if the number of edges colored with any two colors differs by at most one, (ii) \textit{balanced} if for each pair of vertices, among the edges joining the pair, the number of edges of each color differs by at most one from the number of edges of each other color, and (iii) \textit{equitable} if, among the edges incident with each vertex, the number of edges of each color differs by at most one from the number of edges of each other color. 
In \cite{deW71, deW71-2, deW75, deW75-2} de Werra studied balanced equitable edge-coloring of bipartite graphs. The following result is used to prove Theorem \ref{mainth}. 
\begin{theorem} \label{BEE}
Every bipartite graph has a balanced, equitable and equalized $k$-edge-coloring for each $k\in \mathbb N$. 
\end{theorem}  
Here we show that this result is simply a consequence of Nash-Williams lemma. A family $\scr A$ of sets is \textit{laminar} if, for every pair $A, B$ of sets belonging to $\scr A$, either $A\subset B$, or $B\subset A$, or $A\cap B=\varnothing$. 

\begin{lemma}\textup{(Nash-Williams \cite[Lemma 2]{Nash87})}\label{laminarlem}
If $\scr A, \scr B$ are two laminar families of subsets of a finite set $S$, and $n\in \mathbb N$, then there exist a subset $A$ of $S$ such that for every $P\in \scr A \cup \scr B$,  $|A\cap P|\approx |P|/n$. 
\end{lemma}
\noindent{\it Proof of Theorem \ref{BEE}. } 
Let $B$ be a bipartite graph with vertex bipartition $\{V_1, V_2\}$. For $i=1, 2$ define the laminar set $L_i$ to consist of the
following sets of subsets of edges of $B$: (i) The edges between each pair of vertices $v_1 \in V_1$ and $v_2 \in V_2$, (ii) For each $v \in V_i$, the edges incident with $v$, (iii) All the edges in $B$. 
Applying Lemma  \ref{laminarlem} with $n = k$ provides one color class. Remove these
edges then reapply Lemma  \ref{laminarlem}, with $n = k-1$ to get the second class.
Recursively proceeding in this way provides the $k$-edge-coloring of $B$. It is
straightforward to see that this produces the result by observing that 
the edges in subsets defined in (i), (ii) and (iii) guarantee that
the $k$-edge-coloring is balanced, equitable, and equalized
respectively. 
\qed

\section {Applications}
In this section we demonstrate the power of Theorem \ref{mainth}. The results are not new, and many follow from earlier, more restrictive versions of Theorem \ref{mainth}. But the point of this section is to give the reader a feel for how amalgamations can be used. 
\begin{theorem} \textup{(Walecki \cite{L})}\label{knhdthm}
$\lambda K_n$ is Hamiltonian decomposable (with a 1-factor leave, respectively) if and only if $\lambda (n-1)$ is even (odd, respectively).
\end{theorem}
\begin{proof} The necessity is obvious. To prove the sufficiency, let $H$ be a graph with $V(H)=\{v\}$, $\ell(v)=\lambda \binom{n}{2}$ and $\eta(v)=n$ , and let $k=\lfloor\lambda (n-1)/2\rfloor$. Color the loops so that $\ell_{H(j)}(v)=n$, for $1\leq j\leq k$ (and $\ell_{H(k+1)}(v)=n/2$, if $\lambda (n-1)$ is odd).  Applying Theorem \ref{mainth}  completes the proof. 
\end {proof}
The following result is essentially proved in \cite{H2}, but the result is stated in less general terms. 
\begin{theorem} \textup{(Hilton \cite{H2})} \label{knhdthmembd}
A $k$-edge-colored $K_m$ can be embedded into a Hamiltonian decomposition of $K_{m+n}$ (with a 1-factor leave, respectively) if and only if $(m+n-1)$ is even (odd, respectively), $k=\lceil(m+n-1)/2\rceil$, and each color class of $K_m$ (except one color class, say $k$, respectively) is a collection of at most $n$ disjoint paths, (color class $k$ consists of paths of length at most 1, at most $n$ of which are of length 0, respectively), where isolated vertices in each color class  are to be counted as paths of length 0. 
\end{theorem}
\begin{proof} The necessity is obvious. To prove the sufficiency, let $p_i\leq n$ be the number of paths colored $i$, $1\leq i\leq k$.  
Form a graph $H$ by adding a new vertex $u$ to $K_m$ so that $\ell(u)= \binom{n}{2}$, $m(u,v)=n$ for each $v\in V(K_m)$, and $\eta(u)=n$. Color the new edges incident with vertices in $K_m$ so that $d_{H(j)}(v)=2$ for $v\in V(K_m)$, $1\leq j\leq k$
 (if $m+n$ is even, do it so that $d_{H(k)}(v)=1$ for $v\in V(K_m)$; so at most $n$ such edges are incident with $u$ by necessary conditions). Clearly, each color appears on an even number of such edges (except possibly color $k$ when $m+n$ is odd). 
 Color the loops so that $d_{H(j)}(u)=2n$ for $1\leq j\leq k$ (if $m+n$ is even, then the coloring must be so that $d_{H(k)}(u)=n$). 
This is possible since each color appears on $(2n-2p_i)/2\geq 0$ loops. Now applying Theorem \ref{mainth} completes the proof. 
\end {proof}
A similar result can be obtained for embedding $\lambda K_m$ into a Hamiltonian decomposition of $\lambda K_{m+n}$. A more general problem is the following enclosing problem
\begin{question} Find necessary and sufficient conditions for enclosing an edge-colored $\lambda K_m$ into a Hamiltonian decomposition of $\mu K_{m+n}$ for $\lambda < \mu$. 
\end{question}
 An \textit{$(r_1,\ldots, r_k)$-factorization} of a graph $G$ is a partition (decomposition) $\{F_1,\ldots, F_k\}$ of $E(G)$ in which $F_i$ is an $r_i$-factor for $i=1,\ldots, k$. The following is a corollary of a strong result of Johnson \cite{mJohns} in which each color class can have a specified edge-connectivity. A special case of this is proved by Johnstone in \cite{johnst}.  
\begin{theorem} \label{knrkfac}
$\lambda K_n$ is $(r_1,\dots,r_k)$-factorizable if and only if $r_i n$ is even for $1\leq i\leq k$, and $\sum_{i=1}^k r_i=\lambda (n-1)$. Moreover,  for $1\leq i\leq k$ each $r_i$-factor can be guaranteed to be connected if $r_i$ is even.  
\end{theorem}
\begin{proof} The necessity is obvious. To prove the sufficiency, start from the graph $H$ as in the proof of Theorem \ref{knhdthm}, but color the loops so that $\ell_{H(j)}(v)=nr_j/2$ for $1\leq j\leq k$. Then apply Theorem \ref{mainth}.  
\end {proof}
The following result was proved for the special case $r_1=\ldots =r_k=r$ in \cite{AndHilt, RW}. 
\begin{theorem} \label{knrkfactembd}
A $k$-edge-coloring of $K_m$ can be embedded into an $(r_1,\dots\\,r_k)$-factorization of $K_{m+n}$ if and only if $r_i (m+n)$ is even for $1\leq i\leq k$, $\sum_{i=1}^k r_i=m+n-1$, $d_{K_m(i)}(v)\leq r_{\sigma(i)}$ for each $v\in V(K_m)$, $1\leq i\leq k$, and some permutation $\sigma\in S_k$, and $|E(K_m(i))|\geq r_{\sigma(i)}(m-n)/2$. 
\end{theorem}
\begin{proof} The necessity is obvious. To prove the sufficiency, start from the graph $H$ as in the proof of Theorem \ref{knhdthmembd}. Color the new edges incident with vertices in $K_m$ so that $d_{H(j)}(v)=r_{\sigma(j)}$ for  $v\in V(K_m)$, $1\leq j\leq k$. Then color the loops incident with $u$ so that $d_{H(j)}(u)=r_{\sigma(j)} m$ for $1\leq j\leq k$ (the last necessary condition guarantees that the number of required loops is non-negative), and apply Theorem \ref{mainth}. 
 \end {proof}
 \begin{question} Find necessary and sufficient conditions for enclosing an edge-colored $\lambda K_n$ into an $(r_1,\dots,r_k)$-factorization of $\mu K_{m+n}$ for $\lambda \leq \mu$. 
\end{question}
The case $\lambda=\mu$ can be obtained by altering the proof of Theorem \ref{knrkfactembd} slightly. 

 Some of the above results can be easily generalized to complete multipartite graphs. 
\begin{theorem} \textup{(Laskar, Auerbach \cite{LA})}\label{hdknr}
$\lambda K_{n_1,\dots,n_m}$ is Hamiltonian decomposable (with a 1-factor leave, respectively) if and only if $n_1=\dots=n_m:=n$, and $\lambda n(m-1)$ is even (odd, respectively).
\end{theorem}
\begin{proof} The necessity is obvious. To prove the sufficiency,  consider the graph $H:=\lambda n^2K_m$, and $\eta:V(H)\rightarrow \mathbb{N}$ with $\eta(v)=n$ for each $v\in V(H)$. Using Theorem \ref{knrkfac}, find a connected $2n$-factorization of $H$ and apply Theorem \ref{mainth}. 
\end {proof}
Another very nice requirement that one can ask of a Hamiltonian
decomposition of a complete multipartite graph is that it be fair; that
is, in each Hamiltonian cycle, the number of edges between each pair of
parts is within one of the number of edges between each other pair of
parts. This result can be proved by being more careful in the
construction of the edge-coloring of the graph $H$ described in the proof
of Theorem 2.5; ensure that for each color class the number of edges
between each pair of vertices in $H$ is within 1 of the number of edges
between each other pair of vertices (one could think of this color class
as being ``equimultiple"). Leach and Rodger \cite{LRfairhd} used this approach to
prove that
\begin{theorem} \label{fairhdknr}
$K_{n_1,\dots,n_m}$ is fair Hamiltonian decomposable  if and only if $n_1=\dots=n_m:=n$, and $n(m-1)$ is even. 
\end{theorem}
\begin{question} Find necessary and sufficient conditions for enclosing a $k$-edge-colored $\lambda K_{n_1,\dots,n_m}$ into a (fair) Hamiltonian decomposition  of $\mu K_{n'_1,\dots,n'_{m'}}$ for $n_i\leq n'_i$, $1\leq i\leq m\leq m'$, and $\lambda \leq \mu$. 
\end{question}
\begin{theorem} $\lambda K_{n_1,\dots,n_m}$ is $(r_1,\dots,r_k)$-factorizable if and only if $n_1=\dots=n_m:=n$, $r_i nm$ is even for $1\leq i\leq k$, and $\sum_{i=1}^k r_i=\lambda n(m-1)$.
\end{theorem}
\begin{proof} The necessity is obvious. To prove the sufficiency,  use Theorem \ref{knrkfac} to find an  $(nr_1,\dots,nr_k)$-factorization of the graph $H$ described in the proof of Theorem \ref{hdknr}; then apply Theorem \ref{mainth}. 
\end {proof}
\begin{question} Find necessary and sufficient conditions for enclosing a $k$-edge-colored $\lambda K_{n_1,\dots,n_m}$  into an $(r_1,\dots,r_k)$-factorization of $\mu K_{n'_1,\dots,n'_{m'}}$ for $n_i\leq n'_i$, $1\leq i\leq m\leq m'$, and $\lambda \leq \mu$.
\end{question}
The Oberwolfach problem $OP(r_1^{a_1},\dots, r_k^{a_k})$ asks whether or not it is possible to partition the edge set of $K_n$, $n$ odd, or $K_n$ with a 1-factor removed when $n$ is even, into isomorphic 2-factors such that each 2-factor consists of $a_j$ cycles of length $r_j$, $1\leq j\leq k$, and $n=\sum _{j=1}^{k} r_j a_j$. In \cite{HiltonMJohnsonJLM01} some new solutions to the Oberwolfach problem are given  using the amalgamation technique. 

\section {Hamiltonian Decomposition of $K(n_1,\ldots, n_m;\lambda  ,\mu )$ }
Let $K(n_1,\ldots,$ $n_m;\lambda,\mu )$ denote a graph with $m$ parts, the $i^{th}$ part having size $n_i$, in which multiplicity of each pair of vertices in the same part (in different parts) is  $\lambda$ ($\mu $, respectively). When $n_1 = \ldots = n_m=n$, we denote $K(n_1,\ldots, n_m;\lambda,\mu )$ by $K(n^{(m)};\lambda,\mu )$. In \cite{BahRod1}, 
we  settled the existence of Hamiltonian decomposition for $K(n_1,\ldots,n_m;\lambda,\mu)$, 
a graph of particular interest to statisticians, who consider group divisible designs with two associate classes. 

\begin{theorem}\textup{(Bahmanian, Rodger \cite[Theorem 4.3]{BahRod1})}\label{hdka1apl1l2}
Let $m> 1$, $\lambda \geq 0$, and $\mu \geq 1$, with $\lambda   \neq \mu $ be integers. Let $n_1,\ldots,n_m$ be positive integers with $n_1\leq\ldots \leq n_m$, and $n_m\geq 2$. Then $G=K(n_1, \dots, n_m; \lambda , \mu )$ is Hamiltonian decomposable if and only if the following conditions are satisfied:
\begin{itemize}
\item [\textup{(i)}] $n_i=n_j:=n$ for $1\leq i<j\leq m;$
\item [\textup{(ii)}]$\lambda (n-1)+\mu n(m-1)$ is an even integer\textup{;} 
\item [\textup{(iii)}]$\lambda \leq \mu n(m-1).$
\end{itemize}
\end{theorem}
\begin{example} Figure \ref{hdk22221} illustrate a Hamiltonian decomposition of $K(2^{(3)};2,1)$.
\begin{figure}[htbp]
\begin{center}
\scalebox{.40}{ \input {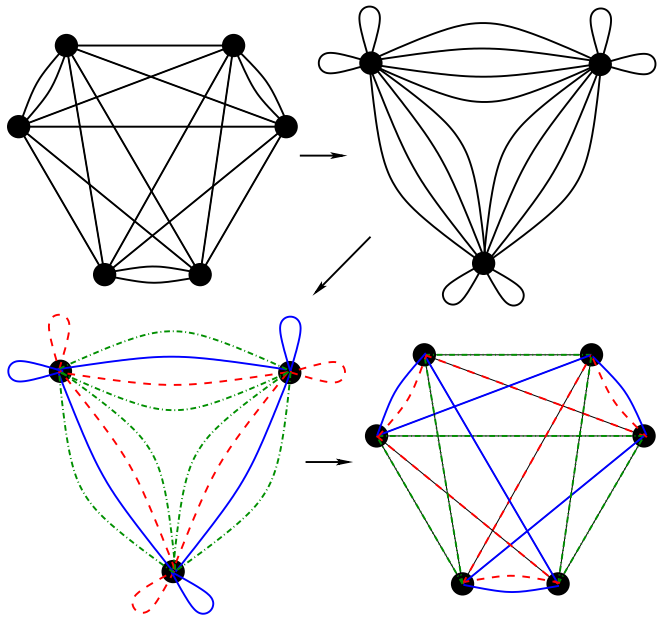} }
\caption{A Hamiltonian Decomposition of $K(2^{(3)};2,1)$  }
\label{hdk22221}
\end{center}
\end{figure} 
\end{example}
In this paper, we solve the companion problem; that is the Hamiltonian decompositions problem for  $K(n_1, \dots, n_m; \lambda , \mu )$ when it is a regular graph of odd degree. The details are provided in order that the reader may become more familiar with the nuances of using amalgamations.  

A graph $G$ is said to be even if all of its vertices have even degree. Let $k$ be a positive integer. We say that $G$ has an evenly-equitable $k$-edge-coloring if $G$ has a $k$-edge-coloring for which,  for each $v\in V(G)$
\begin{itemize}
\item [(i)] $d_{G(i)}(v)$ is even for $1\leq i\leq k$, and 
\item [(ii)] $|d_{G(i)}(v)-d_{G(j)}(v)|\in\{0,2\}$ for $1\leq i,j\leq k$.  
\end{itemize}
We need the following theorem of Hilton \cite{H1}. (It may help to recall that the definition of $k$-edge-coloring allows some color classes to be empty. It is also worth noting that the following theorem is true even if the graph contains loops.)
\begin{theorem}\textup{(Hilton \cite[Theorem 8]{H1})}\label{eveneq}
Each finite even graph has an evenly-equitable $k$-edge-coloring for each positive integer $k$.
\end{theorem} 

Let us first look at some trivial cases:
\begin{itemize}
\item [(i)] If $m=1$, then $G=\lambda K_{n_1}$ which by Theorem \ref{knhdthm}, is decomposable into Hamiltonian cycles and a single 1-factor if and only if $\lambda  (n_1-1)$ is odd. 
\item [(ii)] If $m>1$, $\mu =0$, then $G=\bigcup\limits_{i=1}^m{\lambda K_{n_i}}$. Clearly $G$ is disconnected it does not have any Hamiltonian cycle. 
\item [(iii)] If $n_i=1$ for $1\leq i\leq m$, then $G=\mu  K_m$ which is decomposable into Hamiltonian cycles and a single 1-factor if and only if $\mu  (m-1)$ is odd. 
\item [(iv)] If $\lambda = \mu $, then $G= \lambda  K_{n_1+\dots+n_m}$ which is decomposable into Hamiltonian cycles and a single 1-factor if and only if $\lambda  (\sum\limits_{i=1}^{m}{n_i}-1)$ is even. 
\item [(v)] If $\lambda =0$, and $n_i=n$ for $1\leq i\leq m$, then $G=\mu K_{\underbrace{n,\ldots,n}_m}$ which is decomposable into Hamiltonian cycles and a single 1-factor if and only if $\mu n(m-1)$ is odd (see \cite{LA}). 
\end{itemize}
We exclude the above five cases from our theorem:
\begin{theorem}\label{hd1fka1apl1l2}
Let $m>1$.  Let $n_1,\ldots,n_m$ be positive integers with $n_1\leq\ldots \leq n_m$, and $n_m\geq 2$, and  $\lambda ,\mu \geq 1$ with  $\lambda   \neq \mu $. Then $G=K(n_1, \dots, n_m; \lambda , \mu )$ is decomposable into Hamiltonian cycles and a single 1-factor if and only if the following conditions are satisfied:
\begin{itemize}
\item [\textup{(i)}] $n_i=n_j:=n$ for $1\leq i<j\leq m;$
\item [\textup{(ii)}]$\lambda (n-1)+\mu n(m-1)$ is an odd integer\textup{;} 
\item [\textup{(iii)}]$\lambda \leq  \mu n(m-1)$ if $n\geq 3$, and $\lambda-1 \leq 2\mu (m-1)$ otherwise.
\end{itemize}
\end{theorem}
\begin{proof}
Let $s=\sum\nolimits_{i=1}^{m}{n_i}$. To prove the necessity, suppose $G$ is Hamiltonian decomposable. For $v\in V_i$, $1\leq i\leq m$, we have $d_G(v)=\lambda(n_i-1)+\mu(s-n_i)$. Since $G$ is Hamiltonian decomposable, it is regular. So we have 
$\lambda(n_i-1)+\mu(s-n_i)=\lambda(n_j-1)+\mu(s-n_j)$ for  $1\leq i <  j\leq m$. It follows that  $n_i=n_j:=n$ for  $1\leq i <  j\leq m$. 
So we can assume that $G=K(n^{(m)};\lambda ,\mu )$. Therefore $d_G(v)=\lambda (n-1)+\mu n(m-1)$. Now since $G$ is decomposable into Hamiltonian cycles and a single 1-factor  
$$\lambda (n-1)+\mu n(m-1) \text {    is an odd integer}.$$

By the preceding paragraph, the number of Hamiltonian cycles of $G$ is $\big(\lambda (n-1)+\mu n(m-1)-1\big)/2$. Let us say that an edge is pure if both of its endpoints belong to the same part. Each Hamiltonian cycle passes through every vertex of every part exactly once. Hence each Hamiltonian cycle contains at most $n-1$ pure edges from each part. Since the total number of pure edges in each part is $\lambda \binom{n}{2}$, and a 1-factor contains at most $\lfloor a/2\rfloor$ pure edges from each part, we have 
$$\lambda \binom{n}{2} \leq \frac{(n-1)}{2}\big(\lambda (n-1)+\mu n(m-1)-1\big)+ \lfloor \frac{n}{2}\rfloor.$$
So, 
$$ \frac{\lambda n(n-1)}{2} \leq \frac{(n-1)}{2}\big(\lambda (n-1)+\mu n(m-1)-1\big)+ \lfloor \frac{n}{2}\rfloor.$$
Since $n>1$, it implies that 
$$\lambda n \leq \lambda (n-1)+\mu n(m-1)-1+\frac{2\lfloor \frac{n}{2}\rfloor}{n-1}.$$
It follows that if $n$ is odd, then we have $\lambda \leq \mu n(m-1)$, and if $n>2$ is even, then we have $\lambda \leq \mu n(m-1)+1/(n-1)$, which is equivalent to  $\lambda \leq \mu n(m-1)$. Moreover, if $n=2$, then we have  $\lambda-1 \leq 2 \mu (m-1)$. Therefore conditions (i)-(iii) are necessary. 

To prove the sufficiency, suppose conditions (i)-(iii) are satisfied. We first solve the special case of $n=2$. Since $\lambda+2\mu (m-1)$ is  odd, so is $\lambda$. Also $\lambda-1 \leq 2 \mu (m-1)$. Therefore, by Theorem \ref{hdka1apl1l2}, $K(2^{(m)};\lambda-1, \mu)$ is Hamiltonian decomposable. Adding an edge to each part of $K(2^{(m)};\lambda-1, \mu)$ (which is a 1-factor) will form $K(2^{(m)};\lambda, \mu)$. Thus we obtain a decomposition of $K(2^{(m)};\lambda, \mu)$ into Hamiltonian cycles and a single 1-factor. 
To prove the sufficiency for $n\geq 3$, let $H$ be a graph with $|V(H)|=m, \ell_H(y)=\lambda \binom{n}{2}$ for every $y\in V(H)$, and $m_H(y,z)=\mu n^2$ for every pair $y,z\in V(H)$ and let $\eta$ be a function from $V(H)$ into $\mathbb{N}$ with $\eta(y)=n$ for all $y\in V(H)$. Now define $k=\big(\lambda (n-1)+\mu n(m-1)-1\big)/2$. From (ii), $k$ is an integer. We note that $H$ is $(2k+1)n$-regular. In what follows we shall find an appropriate edge-coloring for $H$ and then we shall apply Theorem \ref{mainth}, to show that $H$ has an $\e$-detachment $G$ in which every color class except one induces a Hamiltonian cycle, the exceptional color class being a -factor. 

Let $H^*$ be the spanning subgraph of $H$ whose edges are the non-loop edges of $H$. It is easy to see that $H^*\cong \mu n^2K_m$. We shall find a $(k+1)$-edge-coloring for $H$. There are two cases to consider, but first we observe that: 
\begin{equation}\label{enoughcyc} \begin{split}
 (n-1)\big(\mu n(m-1)-\lambda\big)\geq 0 & \iff \\
  \mu n(m-1)(n-1)-\lambda (n-1)\geq 0 &\iff \\
    \mu n^2(m-1)-\lambda (n-1)-\mu n(m-1) \geq 0 & \iff \\
  \frac{\mu n^2(m-1)}{2} \geq  \frac{\lambda (n-1)+\mu n(m-1)}{2}.
  \end{split} \end{equation}

\begin{itemize}
\item {\bf Case 1: $n$ is even.}  It follows that $\mu n^2(m-1)$ is even and thus by Theorem \ref{knhdthm}, $H^*$ is decomposable into $\frac{\mu n^2(m-1)}{2}$ Hamiltonian cycles. Now since $n > 1$, and since by (iii) $\mu n(m-1)\geq \lambda $, by (\ref{enoughcyc}) it follows that the number of Hamiltonian cycles in $H^*$ is greater than $k$. Now let $\mathcal{C}_1,\ldots, \mathcal{C}_k$ be $k$ arbitrary Hamilton cycles of a Hamiltonian decomposition of $H^*$. Let $\mathcal{K}^*$ be a (partial) $k$-edge-coloring of $H^*$ such that all edges of each cycle $\mathcal{C}_i$ are colored $i$, for each $1\leq i\leq k$. Now let ${\mathcal L}$ be a spanning subgraph of $H$ in which every vertex is incident with $n/2$ loops (observe that $\lambda \binom{n}{2}\geq n/2$); so the graph ${\mathcal L}$ consists only of loops. Now let $H^{**}$ be the spanning subgraph of $H$ whose edges are all edges in $E(H)\backslash E({\mathcal L})$ that are uncolored in $H^*$. Recall that $H$ is $(2k+1)n$-regular, so for each $v\in V(H^{**})$ we have $d_{H^{**}}(v)=(2k+1)n-2k-2(n/2)=2k(n-1)$. Therefore $H^{**}$ is an even graph and so by Theorem \ref{eveneq} it has an evenly-equitable edge-coloring $\mathcal{K}^{**}$ with $k$ colors $1,\ldots,k$ (Note that we are using the same colors we used to color edges of $H^*$). Therefore for each $j$, $1\leq j\leq k$, and for each $ y\in V(H^{**})$, we have $d_{H^{**}(j)}(y)=2(n-1)k/k=2(n-1)$. Now we can define the $(k+1)$-edges coloring $\mathcal{K}$ for $H$ as below:
$$
\mathcal{K}(e) : = \left \{ \begin{array}{ll}
\mathcal{K}^*(e) & \mbox { if } e\in E(H^*)\backslash E(H^{**}),\\
\mathcal{K}^{**}(e) & \mbox { if } e\in E(H^{**}),\\
k+1  & \mbox { if } e\in E({\mathcal L}).\\
\end{array} \right.
$$
So for each $y\in V(H)$,
$$
d_{H(j)}(y)= \left \{ \begin{array}{ll}
2(n-1)+2=2n & \mbox { if } 1\leq j\leq k,\\
2(n/2)=n & \mbox { if } j=k+1.\\
\end{array} \right. 
$$
\item {\bf Case 2: $n$ is odd.} Since $\lambda (n-1)$ is even, and by (ii), $\lambda (n-1)+\mu n(m-1)$ is odd, it follows that $\mu n(m-1)$ is odd. So $\mu n^2(m-1)$ is odd. Thus by Theorem \ref{knhdthm}, $H^*$ is decomposable into $\big(\mu n^2(m-1)-1\big)/2$ Hamiltonian cycles and a single $1$-factor $F$. 

By \eqref{enoughcyc}, it follows that  $$\frac{\mu n^2(m-1)-1}{2}\geq \frac{\lambda (n-1)+\mu n(m-1)-1}{2}=k.$$
Hence, the number of Hamiltonian cycles in $H^*$ is at least $k$. Now let $\mathcal{C}_1,\ldots, \mathcal{C}_k$ be $k$ arbitrary Hamilton cycles of a Hamiltonian decomposition of $H^*$. Let $\mathcal{K}^*$ be a (partial) $k$-edge-coloring of $H^*$ such that all edges of each cycle $\mathcal{C}_i$ are colored $i$, for each $1\leq i\leq k$, and the single 1-factor $F$ is colored $k+1$. Now let ${\mathcal L}$ be a spanning subgraph of $H$ in which every vertex is incident with $(n-1)/2$ loops (observe that $\lambda \binom{n}{2}\geq (n-1)/2$). Now let $H^{**}$ be the spanning subgraph of $H$ whose edges are all edges in $E(H)\backslash E({\mathcal L})$ that are uncolored in $H^*$. Recall that $H$ is $(2k+1)n$-regular, so for each $v\in V(H^{**})$ we have $d_{H^{**}}(v)=(2k+1)n-2k-1-2(n-1)/2=2k(n-1)$. Therefore $H^{**}$ is an even graph and so by Theorem \ref{eveneq} it has an evenly-equitable edge-coloring $\mathcal{K}^{**}$ with $k$ colors $1,\ldots,k$ (Note that we are using the same colors we used to color edges of $H^*$). Therefore for each $j$, $1\leq j\leq k$, and for each $ y\in V(H^{**})$, we have $d_{H^{**}(j)}(y)=2(n-1)k/k=2(n-1)$. Now we can define the $(k+1)$-edges coloring $\mathcal{K}$ for $H$ as below:
$$
\mathcal{K}(e) : = \left \{ \begin{array}{ll}
\mathcal{K}^*(e) & \mbox { if } e\in E(H^*)\backslash E(H^{**}),\\
\mathcal{K}^{**}(e) & \mbox { if } e\in E(H^{**}),\\
k+1  & \mbox { if } e\in E({\mathcal L}).\\
\end{array} \right.
$$
So for each $y\in V(H)$,
$$
d_{H(j)}(y)= \left \{ \begin{array}{ll}
2(n-1)+2=2n & \mbox { if } 1\leq j\leq k,\\
1+2(n-1)/2=n & \mbox { if } j=k+1.\\
\end{array} \right. 
$$
\end{itemize}
Note that since all edges of each Hamiltonian cycle $\mathcal{C}_j$ are colored $j$, $1\leq j \leq k$, each color class $H(j)$ is connected for $1\leq j \leq k$.
Therefore in both cases, we have a $(k+1)$-edge-colored graph $H$ for which, for each $y,z\in V(H), y\neq z$, $\eta(y)=n\geq 2$, $\ell_H(y)=\lambda \binom{n}{2}$, $m_H(y,z)=\mu n^2$, $d_H(y)=(2k+1)n$, $\omega(H(j))=1$ for each $1\leq j\leq k$, and 
$$
d_{H(j)}(y)= \left \{ \begin{array}{ll}
2n & \mbox { if } 1\leq j\leq k,\\
n & \mbox { if } j=k+1.\\
\end{array} \right. 
$$
Now by Theorem \ref{mainth}, there exists a loopless $\e$-detachment $G^*$ of $H$ in which each $v\in V(H)$ is detached into $v_1,\ldots, v_{\eta(v)}$ such that for each $u,v\in V(H), u\neq v$ the following conditions are satisfied:
\begin{itemize}
  \item $m_{G^*}(u_i,u_{i'})=\lambda \binom{n}{2}/\binom{n}{2}=\lambda $ for  $1\leq i<i'\leq \eta(u)$;
  \item $m_{G^*}(u_i,v_{i'})=\mu n^2/(nn)=\mu $ for $1\leq i\leq \eta(u)$ and $1\leq i'\leq \eta(v)$;
  \item $d_{G^*(j)}(u_i)= \left \{ \begin{array}{ll}
2n/n=2 & \mbox { if } 1\leq j\leq k,\\
n/n=1 & \mbox { if } j=k+1,\\
\end{array} \right.$ for $1\leq i\leq \eta(u)$;
  \item $\omega(G^*(j))=\omega(H(j))=1$ for each $1\leq j\leq k$, since $d_{H(j)}(u)/\eta(u)=2n/n=2$ for $1\leq j\leq k$. 
\end{itemize} 
From the first two conditions it follows that $G\cong K(n^{(m)};\lambda ,\mu )$. The last two conditions tells us that each color class $1\leq j\leq k$ is $2$-regular and connected respectively; that is each color class $1\leq j\leq k$ is a Hamiltonian cycle. Furthermore, the color class $k+1$ is 1-regular. So we obtained a decomposition of $K(n^{(m)};\lambda ,\mu )$ into Hamiltonian cycles and a single 1-factor. 
\end{proof}

\end{document}